\documentclass[11pt,reqno]{amsart}
\usepackage[margin=1.15in]{geometry}
\usepackage[T1]{fontenc}
\usepackage{lmodern}
\usepackage{amsmath,amssymb,amsthm}
\usepackage{microtype}
\usepackage{xcolor}
\definecolor{linkblue}{RGB}{0,51,102}
\usepackage[colorlinks=true,allcolors=linkblue]{hyperref}
\makeatletter
\renewcommand{\@secnumfont}{\bfseries}
\renewcommand{\section}{\@startsection{section}{1}{\z@}%
  {1.35\linespacing\@plus.3\linespacing\@minus.2\linespacing}%
  {.5\linespacing}%
  {\normalfont\large\bfseries\raggedright}}
\renewcommand{\subsection}{\@startsection{subsection}{2}{\z@}%
  {1.05\linespacing\@plus.2\linespacing\@minus.1\linespacing}%
  {.35\linespacing}%
  {\normalfont\normalsize\bfseries\raggedright}}
\renewcommand{\subsubsection}{\@startsection{subsubsection}{3}{\z@}%
  {.8\linespacing\@plus.2\linespacing\@minus.1\linespacing}%
  {.25\linespacing}%
  {\normalfont\normalsize\bfseries\raggedright}}
\makeatother

\makeatletter
\def\@settitle{\begin{center}%
  \baselineskip14\p@\relax
  \bfseries
  \@title
  \end{center}%
}
\def\@setauthors{%
  \begingroup
  \def\thanks{\protect\thanks@warning}%
  \trivlist
  \centering\footnotesize \@topsep30\p@\relax
  \advance\@topsep by -\baselineskip
  \item\relax
  \author@andify\authors
  \def\\{\protect\linebreak}%
  \authors
  \ifx\@empty\contribs
  \else
    ,\penalty-3 \space \@setcontribs
    \@closetoccontribs
  \fi
  \endtrivlist
  \endgroup
}
\def\@setaddresses{\par
  \nobreak \begingroup
  \footnotesize
  \def\author##1{\nobreak\addvspace\bigskipamount}%
  \def\\{\unskip, \ignorespaces}%
  \interlinepenalty\@M
  \def\address##1##2{\begingroup
    \par\addvspace\bigskipamount\indent
    \@ifnotempty{##1}{(\ignorespaces##1\unskip) }%
    {\normalfont\ignorespaces##2}\par\endgroup}%
  \def\curraddr##1##2{\begingroup
    \@ifnotempty{##2}{\nobreak\indent\curraddrname
      \@ifnotempty{##1}{, \ignorespaces##1\unskip}\/:\space
      ##2\par}\endgroup}%
  \def\email##1##2{\begingroup
    \@ifnotempty{##2}{\nobreak\indent\ttfamily##2\par}\endgroup}%
  \def\urladdr##1##2{\begingroup
    \def~{\char`\~}%
    \@ifnotempty{##2}{\nobreak\indent\urladdrname
      \@ifnotempty{##1}{, \ignorespaces##1\unskip}\/:\space
      \ttfamily##2\par}\endgroup}%
  \addresses
  \endgroup
}
\makeatother

\newtheorem{theorem}{Theorem}[section]
\newtheorem*{maintheorem}{Main Theorem}
\newtheorem{lemma}[theorem]{Lemma}
\newtheorem{corollary}[theorem]{Corollary}
\newtheorem{proposition}[theorem]{Proposition}
\numberwithin{equation}{section}
\newcommand{\EE}{\mathcal E}

\title{F\"uredi's Conjecture and a Sharp Strengthening}

\author{Zihao Huang}
\author{Suijie Wang}
\address{School Of Mathematics, Hunan University, Changsha 410082, Hunan, P. R. China}
\email{zihaoh@hnu.edu.cn {\normalfont (Zihao Huang)}\qquad wangsuijie@hnu.edu.cn {\normalfont (Suijie Wang)}}
\thanks{Corresponding author: Suijie Wang.}

\date{}
\subjclass[2020]{Primary 05D05; Secondary 15A75}
\keywords{F\"uredi's conjecture, Bollob\'as set-pairs inequality,
intersecting subspaces, exterior algebra, graded ideals}

\begin{document}

\begin{abstract}
We prove a strengthening of F\"uredi's conjecture on strong Bollob\'as
$t$-systems. For nonnegative integers $t$ and $s$, a family of $m\ge2$ pairs
of finite sets $(A_i,B_i)$ satisfying $|A_i\cap B_i|\le t$ and
$|A_i\cap B_j|>t+s$ for all $i\ne j$ satisfies
\[
 \sum_{i=1}^{m}
 \frac{\binom{|B_i|-t+s}{s}}
      {\binom{|A_i|+|B_i|-2t}{|A_i|-t}}\le1.
\]
When $s=0$, this is precisely F\"uredi's conjectured weighted inequality.
The proof is based on a local growth inequality for graded exterior ideals,
a basis-exchange construction, and a common-section reduction. The same
method gives the corresponding inequality for subspaces over arbitrary
fields. We also determine a best possible upper bound for the F\"uredi
weight sum, construct equality examples with an arbitrary prescribed number
of pairs, and derive further inequalities from two-sided basis exchanges.
\end{abstract}

\maketitle
\markboth{Zihao Huang And Suijie Wang}{F\"uredi's Conjecture}

\section{Introduction}

Bollob\'as' set-pairs inequality is one of the basic intersection
inequalities in extremal set theory. F\"uredi~\cite{Furedi} proved the corresponding uniform bound under the
weaker skew cross-intersection condition and conjectured the weighted inequality
for strong
Bollob\'as $t$-systems of nonuniform sizes. Heged\"us~\cite{Heg} later
proved the conjectured inequality under the additional assumption that
$|A_i|+|B_i|$ is independent of $i$. We remove this restriction and prove a stronger inequality that records the
gap between the diagonal and cross-intersection conditions. The same statement
holds for subspaces over arbitrary fields.

\begin{maintheorem}
Let $t,s$ be nonnegative integers and let $m\ge2$. Suppose that one of the
following holds.

\begin{enumerate}
\item[\textup{(i)}] The pairs $(A_i,B_i)$, $1\le i\le m$, are finite sets
satisfying
\[
 |A_i\cap B_i|\le t,\qquad |A_i\cap B_j|>t+s\quad(i\ne j),
\]
and $a_i=|A_i|$, $b_i=|B_i|$.

\item[\textup{(ii)}] The pairs $(A_i,B_i)$, $1\le i\le m$, are subspaces of
a finite-dimensional vector space over an arbitrary field satisfying
\[
 \dim(A_i\cap B_i)\le t,\qquad
 \dim(A_i\cap B_j)>t+s\quad(i\ne j),
\]
and $a_i=\dim A_i$, $b_i=\dim B_i$.
\end{enumerate}

Then, in either case,
\begin{equation}\label{eq:main-bound}
 \sum_{i=1}^{m}
 \frac{\binom{b_i-t+s}{s}}
      {\binom{a_i+b_i-2t}{a_i-t}}\le1.
\end{equation}
\end{maintheorem}

For finite sets, the case $s=0$ is precisely F\"uredi's conjectured
weighted inequality, while the subspace statement gives its unrestricted
analogue over arbitrary fields. For $s>0$, the stronger requirement $|A_i\cap B_j|\ge t+s+1$
produces the additional factor $\binom{b_i-t+s}{s}$ in the main inequality.
In the set setting, we call a family satisfying the hypotheses above a
\emph{strong Bollob\'as $(t,s)$-system}; when $s=0$, this is a strong
Bollob\'as $t$-system in the terminology of Heged\"us~\cite[p.~244]{Heg}.
The Main Theorem remains valid after interchanging $A_i$ and $B_i$.
Corollary~\ref{cor:sharp-constant} determines the optimal bound for the
original reciprocal-binomial sum and characterizes equality when $s\ge1$.

The principal algebraic ingredient is a local growth inequality for graded
exterior ideals satisfying inclusions determined by the subspace pairs. Its
lower bound contains an explicit contribution from each pair, and summation
over the exterior degrees produces the weighted inequality. The proof uses a
projection--colon recursion, while basis exchange gives the strengthened
inequalities for $s>0$ and a common section reduces the general case to
$t=0$. Extension of scalars treats finite fields.

For $t=s=0$, the set statement reduces to the classical Bollob\'as
set-pairs inequality~\cite{Bollobas}, which admits a standard proof by
counting permutations. Frankl~\cite{Frankl} established the fixed-size
bound under the weaker skew condition $A_i\cap B_j\ne\varnothing$ for
$i<j$. For general $t$, when $s=0$ and $|A_i|=a$, $|B_i|=b$ for every
$i$, the Main Theorem recovers F\"uredi's bound
$m\le\binom{a+b-2t}{a-t}$. Thus, already for $s=0$, the Main Theorem removes the constant-sum
hypothesis in the previously known weighted result of Heged\"us~\cite{Heg},
while $s>0$ gives a further strengthening. The reciprocal binomial weights
appearing here are of the same type as those in the classical
Lubell--Yamamoto--Meshalkin inequality~\cite{Lubell,Meshalkin,Yamamoto}.

A parallel line of work concerns subspace analogues of Bollob\'as-type
inequalities, where the exterior-algebra method goes back to
Lov\'asz~\cite{Lovasz}; see also Babai and Frankl~\cite{BF}. Scott and
Wilmer~\cite[Theorem~4.5]{SW} proved a weighted skew inequality for real
subspaces under the opposite monotonicity assumptions
\[
 a_1\le a_2\le\cdots\le a_m,
 \qquad b_1\ge b_2\ge\cdots\ge b_m.
\]
Their Corollary~4.9 gives the symmetric case when $a_i+b_i$ is constant,
since a simultaneous relabeling then supplies the required ordering.
Symmetric cross-intersections alone do not imply these monotonicity
assumptions. Over $\mathbb R$, the unrestricted case $t=s=0$ was formulated
as Conjecture~2, Eq.~(11), in Heged\"us~\cite{HegPreprint}, and its general
$t$ counterpart as Conjecture~3 there; Theorem~3.1 of that paper proves the
equivalence of the two assertions by taking a common section.
Proposition~\ref{thm:subspaces} proves the case $t=s=0$ over arbitrary
fields, and the Main Theorem extends it to arbitrary $t$ and $s$ over every
field. Under the symmetric cross-intersection condition, the Main Theorem
requires neither a constant dimension sum nor an ordering of the dimensions.

Zhu~\cite{Zhu} studied bounded-size set-pair families satisfying
$|A_i|\le r$, $|B_i|\le u$, $|A_i\cap B_i|\le t$, and
$|A_i\cap B_j|>\ell$ for $i<j$, where $\ell\ge t$. The Main Theorem
concerns instead a nonuniform weighted inequality, its arbitrary-field
subspace analogue, and the two-sided basis-exchange inequalities of
Section~\ref{subsec:two-sided} under the symmetric cross-intersection
condition. Tuza~\cite{Tuza87} considered a different relaxation of the
cross-intersection condition. Further variants with prescribed intersection
conditions and $t$-intersecting variants appear in
\cite{Talbot,KKK}, weighted and nonuniform skew variants in
\cite{HF,WLLF,Yue}, and multipart or related subspace variants in
\cite{Alon,LFW,YuEtAl}; see also Tuza~\cite{TuzaSurvey} and Frankl and
Tokushige~\cite[Chapter~26]{FT}.

The paper is organized as follows. Section~\ref{sec:exterior} gives the
exterior-algebra preliminaries. Section~\ref{sec:local} proves the local
growth inequality, and Section~\ref{sec:main-proof} derives the Main Theorem
using basis exchange and a common-section reduction. Section~\ref{sec:consequences}
treats the sharp bound for the F\"uredi weight sum, equality constructions,
and further exchange inequalities.

\section{Exterior algebra preliminaries}\label{sec:exterior}

Let $V$ be an $n$-dimensional vector space over an arbitrary field $\mathbb F$.
The exterior algebra of $V$ is
\[
 \EE(V)=\bigoplus_{j=0}^{n}\Lambda^jV.
\]
If $U\subseteq V$ has dimension $d$, then $\dim\Lambda^dU=1$, and the
exterior product of any basis of $U$ is a nonzero generator of
$\Lambda^dU$.
For $u\in\EE(V)$ and $A\subseteq\EE(V)$, write
$
 u\wedge A=\{u\wedge a:a\in A\}.
$

An ideal $M\subseteq\EE(V)$ is \emph{graded} if
\[
 M=\bigoplus_{j=0}^{n}M_j,\qquad M_j=M\cap\Lambda^jV.
\]
We set $\Lambda^jV=0$ and $M_j=0$ for $j\notin\{0,\ldots,n\}$.
For a graded ideal $M$ and a vector $e\in V$, define the colon ideal
\[
 M:e=\{x\in\EE(V):e\wedge x\in M\}.
\]
It is again a graded ideal.
For a subspace $S\subseteq V$, the ideal of $\EE(V)$ generated by $S$ is
\[
(S):=\operatorname{span}_{\mathbb F}
\{\,s\wedge a:s\in S,\ a\in \EE(V)\,\}.
\]

\begin{lemma}\label{lem:annihilator}
Let $S\subseteq V$ be a subspace of dimension $d$. Choose a basis
$s_1,\ldots,s_d$ of $S$ and put
$v_S=s_1\wedge\cdots\wedge s_d$, with $v_S=1$ if $d=0$. Then
\begin{equation}\label{eq:ann}
 \{x\in\EE(V):x\wedge v_S=0\}=(S).
\end{equation}
\end{lemma}

\begin{proof}
Extend $s_1,\ldots,s_d$ to a basis $s_1,\ldots,s_n$ of $V$.
The products
\[
 s_{i_1}\wedge\cdots\wedge s_{i_j},\qquad i_1<\cdots<i_j,
\]
together with $1$ in degree zero, form a basis of $\EE(V)$; we call them
exterior basis monomials. The exterior basis monomials containing at least
one of $s_1,\ldots,s_d$ form a basis of $(S)$, and each has zero wedge
product with $v_S$. Every remaining basis monomial wedges with $v_S$ to a
nonzero scalar multiple of a basis monomial, and distinct remaining
monomials yield distinct basis monomials. This proves~\eqref{eq:ann}.
\end{proof}

\begin{lemma}\label{lem:colon}
Let $S\subseteq V$ be a subspace and let $e\in V\setminus S$. Then
\begin{equation}\label{eq:colon-linear}
 (S):e=(S+\mathbb F e).
\end{equation}
\end{lemma}

\begin{proof}
Choose a basis $s_1,\ldots,s_d$ of $S$ and extend it to a basis of
$V$ containing $e$. The
condition $e\wedge x\in(S)$ holds precisely when every monomial
occurring in $x$ either contains one of $s_1,\ldots,s_d$ or contains
$e$. These are exactly the monomials spanning $(S+\mathbb F e)$,
which proves~\eqref{eq:colon-linear}.
\end{proof}

For a graded ideal $M\subseteq\EE(V)$ and an integer $k$, put
\[
 \lambda_n(M,k)=k\dim M_k-(n-k+1)\dim M_{k-1}.
\]
For $1\le k\le n$, the identity
$k\binom nk=(n-k+1)\binom n{k-1}$ gives
\begin{equation}\label{eq:normalized-difference}
 \frac{\lambda_n(M,k)}{k\binom nk}
 =\frac{\dim M_k}{\binom nk}
  -\frac{\dim M_{k-1}}{\binom n{k-1}}.
\end{equation}
Here $\binom nk=\dim\Lambda^kV$, so the right-hand side is
the difference of the normalized dimensions in consecutive degrees.

For a linear map $\pi:V\to W$, we use the same symbol for the
induced graded algebra homomorphism $\pi:\EE(V)\to\EE(W)$.
It is defined by $\pi(1)=1$ and
\[
 \pi(v_1\wedge\cdots\wedge v_j)
 =\pi(v_1)\wedge\cdots\wedge\pi(v_j)
 \qquad\forall\,v_1,\ldots,v_j\in V,
\]
and extended linearly. For each $j$, $\pi_j:=\left.\pi\right|_{\Lambda^jV}$ is a linear map from $\Lambda^jV$ to $\Lambda^jW$.
Given a graded subspace $X\subseteq\EE(V)$, write
$\pi X=\pi(X)$ and $\pi_jX=\pi_j(X_j)$ for convenience.

\begin{lemma}\label{lem:projection}
Let $e\in V$ be nonzero and write $V=W\oplus\mathbb F e$. Let $\pi:\EE(V)\to\EE(W)$ be the graded algebra homomorphism induced by the projection $\pi:V\to W$ defined by $\pi(w+ae)=w$ for $w\in W$ and $a\in\mathbb F$. Let $M\subseteq\EE(V)$ be a graded ideal. Then $\pi M$ and $\pi(M:e)$ are graded ideals of $\EE(W)$ with $\pi M\subseteq\pi(M:e)$. Moreover, for every integer $k$,
\begin{equation}\label{eq:hilbert-split}
 \dim M_k=\dim\pi_kM+\dim\pi_{k-1}(M:e),
\end{equation}
and consequently
\begin{equation}\label{eq:flux-split}
 \lambda_n(M,k)\ge \lambda_{n-1}(\pi M,k)
 +\lambda_{n-1}\bigl(\pi(M:e),k-1\bigr).
\end{equation}
\end{lemma}

\begin{proof}
The ideal assertions follow from the surjectivity of $\pi$ and $M\subseteq M:e$. Since $V=W\oplus\mathbb F e$,
\[
 \Lambda^kV=\Lambda^kW\oplus(e\wedge\Lambda^{k-1}W).
\]
The linear map $\pi_k$ is the identity on the first summand and vanishes
on the second. Consequently,
\[
 \ker\bigl(\pi_k|_{M_k}:M_k\to\pi_kM\bigr)
 =M_k\cap(e\wedge\Lambda^{k-1}W).
\]
Multiplication by $e$ restricts to a linear isomorphism
\[
 \pi_{k-1}(M:e)
 \xrightarrow{\ x\mapsto e\wedge x\ }
 M_k\cap(e\wedge\Lambda^{k-1}W).
\]
Indeed, take $x\in (M:e)_{k-1}$ and set $y=\pi_{k-1}x$. Then
$x-y\in e\wedge\Lambda^{k-2}W$, so $e\wedge y=e\wedge x\in M_k$.
Conversely, if $y\in\Lambda^{k-1}W$ and $e\wedge y\in M_k$, then
$y\in M:e$, and hence $y=\pi_{k-1}y\in\pi_{k-1}(M:e)$. Since
multiplication by $e$ is injective on $\Lambda^{k-1}W$, the claim follows. Thus
\[
 0\longrightarrow\pi_{k-1}(M:e)
 \xrightarrow{\ x\mapsto e\wedge x\ }M_k
 \xrightarrow{\ \pi_k\ }\pi_kM\longrightarrow0
\]
is exact. Taking dimensions proves~\eqref{eq:hilbert-split}.
Applying this identity in degrees $k$ and $k-1$, together with
$\pi_{k-1}M\subseteq\pi_{k-1}(M:e)$, gives~\eqref{eq:flux-split}.
\end{proof}

\section{A local growth principle for graded exterior ideals}\label{sec:local}

The following theorem is the principal algebraic estimate used in the
proof of Proposition~\ref{thm:subspaces}.

We use the conventions $\sum_{i\in\varnothing} I_i=0$ and
$\bigcap_{i\in\varnothing} I_i=\EE(V)$. We also set $\binom ab=0$ for
integers $a\ge0$ and $b\notin\{0,\ldots,a\}$.

\begin{theorem}\label{thm:sandwich}
Let $m\ge0$, and let $(A_i,B_i)$, $1\le i\le m$, be pairs of
subspaces of an $n$-dimensional vector space $V$ over an arbitrary
field $\mathbb F$, with positive
dimensions $a_i=\dim A_i$ and $b_i=\dim B_i$, such that
\[
 A_i\cap B_i=\{0\},\qquad A_i\cap B_j\ne\{0\}\quad(i\ne j).
\]
For any choice of nonzero $u_i\in\Lambda^{a_i}A_i$, if a graded ideal
$M\subseteq\EE(V)$ satisfies
\begin{equation}\label{eq:sandwich}
 \sum_{i=1}^m u_i\wedge(B_i)
 \ \subseteq\ M\ \subseteq\ \bigcap_{i=1}^m(B_i),
\end{equation}
then, for $0\le k\le n+1$,
\begin{equation}\label{eq:sandwich-bound}
 \lambda_n(M,k)\ge
 \sum_{i=1}^m b_i\binom{n-a_i-b_i}{k-a_i-1}.
\end{equation}
\end{theorem}

\begin{proof}
We begin by recording a consequence of the cross-intersection condition that will be used repeatedly. For $i\ne j$, choose $0\ne w\in A_i\cap B_j$ and extend $w$ to a basis $w,x_2,\ldots,x_{a_i}$ of $A_i$. Since $u_i\in \Lambda^{a_i}A_i$ is nonzero, there is a nonzero scalar $c$ such that
\[
u_i=cw\wedge x_2\wedge\cdots\wedge x_{a_i}.
\]
As $w\in B_j$, this gives $u_i\in (B_j)$.

Assume first that $\mathbb F$ is infinite. We prove the assertion for all
$k$ simultaneously by induction on the pair $(n,m)$, ordered
lexicographically with $n$ as the first coordinate. Thus the induction
hypothesis is available whenever the ambient dimension decreases, or when
the ambient dimension is unchanged and the number of pairs decreases.
If $n=0$, the positivity of $a_i,b_i$
forces $m=0$, and the assertion is immediate. In every dimension,
both sides of~\eqref{eq:sandwich-bound} vanish for $k=0$ and $k=n+1$.
We may therefore assume $n\ge1$ and fix $1\le k\le n$.
The condition $A_i\cap B_i=\{0\}$ gives $a_i+b_i\le n$.

\smallskip
\noindent\emph{Case 1: $a_i+b_i<n$ for every $i$.}

Choose a nonzero vector $e$ outside all the subspaces $A_i+B_i$.
Such a choice is possible because a finite union of proper linear
subspaces cannot cover $V$ over an infinite field. When $m=0$, take
any nonzero $e$.
Write $V=W\oplus\mathbb F e$, and let $\pi:V\to W$ be the projection
with kernel $\mathbb F e$. It induces a surjective graded algebra homomorphism
$\pi:\EE(V)\to\EE(W)$, denoted by the same symbol. Since
$e\notin A_i+B_i$ for every $i$, the
pairs $\bigl(\pi(A_i),\pi(B_i)\bigr)$ satisfy
\[
 \dim\pi(A_i)=a_i,\qquad \dim\pi(B_i)=b_i,\qquad
 \pi(A_i)\cap\pi(B_i)=\{0\},
\]
and
\[
 \pi(A_i)\cap\pi(B_j)\ne\{0\}\qquad(i\ne j).
\]

Since $\pi\bigl((B_i)\bigr)=\bigl(\pi(B_i)\bigr)$, applying $\pi$ to the lower inclusion
in~\eqref{eq:sandwich} and using Lemma~\ref{lem:projection} gives
\[
 \sum_i\pi(u_i)\wedge\bigl(\pi(B_i)\bigr)
 \subseteq\pi M\subseteq\pi(M:e).
\]
Here $\pi(u_i)$ is a nonzero element of $\Lambda^{a_i}\pi(A_i)$, and
\[
 \pi(M:e)\subseteq\pi\bigl((B_i):e\bigr)
 =\pi\bigl((B_i+\mathbb F e)\bigr)=\bigl(\pi(B_i)\bigr).
\]
Since $\dim W=n-1$, the induction hypothesis applies to $\pi M$ in
degree $k$ and to $\pi(M:e)$ in degree $k-1$.
By~\eqref{eq:flux-split},
\[
 \lambda_n(M,k)
 \ge\sum_i b_i\left[
       \binom{n-1-a_i-b_i}{k-a_i-1}
       +\binom{n-1-a_i-b_i}{k-a_i-2}\right]
 =\sum_i b_i\binom{n-a_i-b_i}{k-a_i-1},
\]
where the last equality is Pascal's identity.

\smallskip
\noindent\emph{Case 2: $a_i+b_i=n$ for at least one $i$.}

First, we construct $M^+$. Put
\[
 P=\{i:a_i+b_i=n\},\qquad
 Q=\{1,\ldots,m\}\setminus P.
\]
Then $P\ne\varnothing$ and $V=A_i\oplus B_i$ for every $i\in P$.
On the right-hand side of~\eqref{eq:sandwich-bound}, precisely those
indices $i\in P$ with $a_i=k-1$ contribute. Set
\[
 P_k=\{i\in P:a_i=k-1\},\qquad
 M^+=M+\sum_{i\in P_k}u_i\wedge\EE(V).
\]
For $i\in P_k$ and $j\in Q$, the cross-intersection condition gives
$u_i\in(B_j)$, and hence $u_i\wedge\EE(V)\subseteq(B_j)$.
Together with~\eqref{eq:sandwich}, this yields
\[
 \sum_{j\in Q}u_j\wedge(B_j)
 \subseteq M^+\subseteq\bigcap_{j\in Q}(B_j).
\]
Since $P\ne\varnothing$, we have $|Q|<m$. Thus the induction hypothesis, in the same ambient dimension $n$ and with
fewer pairs, applies to $M^+$ and gives
\[
 \lambda_n(M^+,k)\ge
 \sum_{j\in Q}b_j\binom{n-a_j-b_j}{k-a_j-1}.
\]

Second, we compare $M$ and $M^+$.
For $i\in P_k$, since $u_i\wedge A_i=0$, the lower inclusion
in~\eqref{eq:sandwich} gives
$u_i\wedge V=u_i\wedge B_i\subseteq M_k$. Therefore
\[
 M^+_{k-1}=M_{k-1}+\operatorname{span}\{u_i:i\in P_k\},
 \qquad M^+_k=M_k.
\]
We show that the images of $u_i$, $i\in P_k$, are linearly independent in $\Lambda^{k-1}V/M_{k-1}$.
Suppose $\sum_{i\in P_k}c_i u_i\in M_{k-1}$, and fix $j_0\in P_k$.
Let $v_{j_0}\in\Lambda^{b_{j_0}}B_{j_0}$ be the exterior product of a basis of
$B_{j_0}$. Since $M\subseteq(B_{j_0})$, Lemma~\ref{lem:annihilator} gives
\[
 0=\left(\sum_{i\in P_k}c_i u_i\right)\wedge v_{j_0}
   =\sum_{i\in P_k}c_i\,u_i\wedge v_{j_0}.
\]
For $i\ne j_0$, the observation at the beginning of the proof gives
$u_i\in(B_{j_0})$, hence $u_i\wedge v_{j_0}=0$. Thus
$0=c_{j_0}u_{j_0}\wedge v_{j_0}$. Since $V=A_{j_0}\oplus B_{j_0}$,
$u_{j_0}\wedge v_{j_0}$ is a nonzero generator of $\Lambda^nV$; hence
$c_{j_0}=0$. As $j_0$ was arbitrary,
\[
 M_{k-1}\cap\operatorname{span}\{u_i:i\in P_k\}=\{0\},
 \qquad
 \dim\operatorname{span}\{u_i:i\in P_k\}=|P_k|.
\]
Hence $\dim M^+_{k-1}=\dim M_{k-1}+|P_k|$. Together with
$M^+_k=M_k$, this gives
\begin{equation}\label{eq:saturated-flux}
 \lambda_n(M,k)=\lambda_n(M^+,k)+(n-k+1)|P_k|.
\end{equation}

Finally, we complete the proof by combining
\eqref{eq:saturated-flux} with the preceding induction bound. For $i\in P_k$,
we have $b_i=n-a_i=n-k+1$, whereas every $i\in P\setminus P_k$ contributes
zero to the right-hand side of~\eqref{eq:sandwich-bound}. Therefore,
\begin{align*}
 \lambda_n(M,k)
 &\ge \sum_{j\in Q}b_j
       \binom{n-a_j-b_j}{k-a_j-1}
       +(n-k+1)|P_k|\\
 &=\sum_{j\in Q}b_j\binom{n-a_j-b_j}{k-a_j-1}
   +\sum_{i\in P}b_i\binom{0}{k-a_i-1}\\
 &=\sum_{i=1}^m b_i\binom{n-a_i-b_i}{k-a_i-1}.
\end{align*}
This proves~\eqref{eq:sandwich-bound} over infinite fields.

It remains to consider the case where $\mathbb F$ is finite. Let
$\overline{\mathbb F}$ be an algebraic closure of $\mathbb F$, and set
$V_{\overline{\mathbb F}}=\overline{\mathbb F}\otimes_{\mathbb F}V$ and
$M_{\overline{\mathbb F}}=\overline{\mathbb F}\otimes_{\mathbb F}M$; for
$S\subseteq V$, write $S_{\overline{\mathbb F}}=\overline{\mathbb F}\otimes_{\mathbb F}S$.
By~\cite[Lemma~2.4]{HF}, scalar extension preserves dimensions and
intersections. Exterior powers and wedge products commute with scalar
extension, and hence the ideal inclusions in~\eqref{eq:sandwich} are
preserved as well; compare also the proof of~\cite[Theorem~1.8]{HF}.
Thus the extended pairs
$(A_{i,\overline{\mathbb F}},B_{i,\overline{\mathbb F}})$ and the graded ideal
$M_{\overline{\mathbb F}}$ satisfy the same hypotheses over the infinite field
$\overline{\mathbb F}$, while
$\dim_{\overline{\mathbb F}}(M_{\overline{\mathbb F}})_j=\dim_{\mathbb F}M_j$.
The infinite-field case therefore gives the required inequality over
$\mathbb F$.
\end{proof}

\section{Proof of the Main Theorem}\label{sec:main-proof}\label{sec:proof}

This section proves the Main Theorem.
First, Theorem~\ref{thm:sandwich} is summed over the exterior degrees to
obtain Proposition~\ref{thm:subspaces}. Next, Lemma~\ref{lem:exchange}, specialized to $(p,q)=(s,0)$, yields the
case $t=0$ for arbitrary $s$.
Finally, Lemma~\ref{lem:common-section}
reduces the general subspace theorem to $t=0$, after which the set theorem
follows from the coordinate-subspace representation.

\subsection{The basic subspace inequality}
\label{subsec:zero-subspaces}

\begin{lemma}\label{lem:sum}
For positive integers $a,b$ with $a+b\le n$,
\[
 \sum_{k=1}^{n}
 \frac{b\binom{n-a-b}{k-a-1}}{k\binom nk}
 =\frac{1}{\binom{a+b}{a}}.
\]
\end{lemma}

\begin{proof}
For $a+1\le k\le n-b+1$, cancellation of factorials gives
\[
 \frac{b\binom{n-a-b}{k-a-1}}{k\binom nk}
 =\frac{\binom{k-1}{a}\binom{n-k}{b-1}}
        {\binom{a+b}{a}\binom n{a+b}}.
\]
For all other $k\in\{1,\ldots,n\}$, both sides are zero. The sum
of the numerators on the right is $\binom n{a+b}$: an $(a+b)$-element
subset of $\{1,\ldots,n\}$ whose $(a+1)$-st smallest element is $k$
is obtained by choosing $a$ elements below $k$ and $b-1$ above it.
Summing over $k$ proves the identity.
\end{proof}

We first deduce the basic subspace inequality from
Theorem~\ref{thm:sandwich}.

\begin{proposition}\label{thm:subspaces}
Let $(A_i,B_i)$, $1\le i\le m$, be pairs of nonzero subspaces of a
finite-dimensional vector space over an arbitrary field.
Set $a_i=\dim A_i$ and
$b_i=\dim B_i$. If
\[
 A_i\cap B_i=\{0\},\qquad A_i\cap B_j\ne\{0\}\quad(i\ne j),
\]
then
\[
 \sum_{i=1}^{m}\frac{1}{\binom{a_i+b_i}{a_i}}\le1.
\]
\end{proposition}

\begin{proof}
The empty family is immediate, so assume $m\ge1$.
Let $V$ be the ambient space and $n=\dim V$. Choose arbitrary nonzero
$u_i\in\Lambda^{a_i}A_i$ and set $M=\bigcap_{i=1}^m(B_i)$.
For $i\ne j$, we have $u_i\in(B_j)$, hence
$u_i\wedge(B_i)\subseteq(B_j)$; also
$u_i\wedge(B_i)\subseteq(B_i)$. Thus
$u_i\wedge(B_i)\subseteq M$ for every $i$, so
Theorem~\ref{thm:sandwich} applies. Since each $B_i\ne0$,
$M_0=0$ and $M_n=\Lambda^nV$. Hence, summing
\eqref{eq:normalized-difference} over $1\le k\le n$ gives
\[
 1=\sum_{k=1}^{n}\frac{\lambda_n(M,k)}{k\binom nk}.
\]
Applying Theorem~\ref{thm:sandwich} and Lemma~\ref{lem:sum}, we obtain
\begin{equation*}
 1\ge\sum_{i=1}^{m}\sum_{k=1}^{n}
       \frac{b_i\binom{n-a_i-b_i}{k-a_i-1}}{k\binom nk}
  =\sum_{i=1}^{m}\frac{1}{\binom{a_i+b_i}{a_i}}.
 \qedhere
\end{equation*}
\end{proof}

\subsection{A basis-exchange inequality}
\label{subsec:exchanges}

For nonnegative integers $p,q$ and positive integers $a,b$, put
\begin{equation}\label{eq:exchange-weight}
 H_{p,q}(a,b)=
 \sum_{\ell=0}^{\min(p,q)}\binom{a}{p-\ell}\binom{b}{q-\ell}.
\end{equation}
The next lemma applies Proposition~\ref{thm:subspaces} after exchanging
basis vectors between the two members of each pair. The condition
$\dim(A_i\cap B_j)>s$ ensures that all distinct pairs in the enlarged
family remain cross-intersecting.

\begin{lemma}\label{lem:exchange}
Let $s\ge0$ be an integer, and let $(A_i,B_i)$, $1\le i\le m$, be
pairs of subspaces of a finite-dimensional vector space over an arbitrary
field, with $m\ge2$. Suppose that
\[
 A_i\cap B_i=\{0\},\qquad
 \dim(A_i\cap B_j)>s\quad(i\ne j).
\]
Writing $a_i=\dim A_i$ and $b_i=\dim B_i$, we have, for all nonnegative
integers $p,q$ with $p+q\le s$,
\begin{equation}\label{eq:exchange-bound}
 \sum_{i=1}^{m}
 \frac{H_{p,q}(a_i,b_i)}{\binom{a_i+b_i}{a_i-p+q}}\le1.
\end{equation}
\end{lemma}

\begin{proof}
The cross-intersection conditions give $a_i,b_i\ge s+1$. For each $i$,
choose bases $E_i$ of $A_i$ and $F_i$ of $B_i$. Since
$A_i\cap B_i=\{0\}$, the set $E_i\cup F_i$ is linearly independent.
For every $\ell$ with $0\le\ell\le\min(p,q)$ and all subsets
$S\subseteq E_i$, $T\subseteq F_i$ satisfying
\[
 |S|=p-\ell,\qquad |T|=q-\ell,
\]
define
\begin{equation}\label{eq:exchange-pairs}
 \begin{split}
 C_{i,S,T}&=\operatorname{span}\bigl((E_i\setminus S)\cup T\bigr),\\
 D_{i,S,T}&=\operatorname{span}\bigl((F_i\setminus T)\cup S\bigr).
 \end{split}
\end{equation}
The first subspace determines $S,T$, and $\ell$ is determined by $|S|$.
Consequently, the $i$th original pair gives exactly $H_{p,q}(a_i,b_i)$
distinct pairs. Each satisfies
\[
 C_{i,S,T}\cap D_{i,S,T}=\{0\},\qquad
 \dim C_{i,S,T}=a_i-p+q,\quad
 \dim D_{i,S,T}=b_i+p-q.
\]
Both dimensions are positive because $a_i,b_i\ge s+1$ and $p+q\le s$.

We verify the cross-intersection conditions for the enlarged family.
For a fixed $i$, two distinct choices $(S,T)$ and $(S',T')$ give distinct
subsets $(E_i\setminus S)\cup T$ and $(E_i\setminus S')\cup T'$ of
$E_i\cup F_i$ having the same cardinality. There is therefore a basis
vector in the first subset but not the second. This vector belongs to
$C_{i,S,T}\cap D_{i,S',T'}$, so this intersection is nonzero.

Now let $i\ne j$. The subspaces
$\operatorname{span}(E_i\setminus S)$ and
$\operatorname{span}(F_j\setminus T')$ have codimensions $|S|$ in $A_i$
and $|T'|$ in $B_j$, respectively. Restricting to these subspaces reduces
the dimension of $A_i\cap B_j$ by at most $|S|+|T'|$. Hence
\[
 \begin{split}
 \dim(C_{i,S,T}\cap D_{j,S',T'})
 &\ge\dim(A_i\cap B_j)-|S|-|T'|\\
 &\ge s+1-p-q\ge1.
 \end{split}
\]
Thus all distinct members of the enlarged family satisfy the symmetric
cross-intersection condition. Applying Proposition~\ref{thm:subspaces}
and grouping the terms arising from each original pair gives
\[
 1\ge\sum_{i=1}^{m}
 \frac{H_{p,q}(a_i,b_i)}
      {\binom{(a_i-p+q)+(b_i+p-q)}{a_i-p+q}},
\]
which is~\eqref{eq:exchange-bound}.
\end{proof}

Taking $p=s$ and $q=0$ in Lemma~\ref{lem:exchange}, and using
\begin{equation}\label{eq:one-sided-identity}
 \frac{\binom{a}{s}}{\binom{a+b}{a-s}}
 =\frac{\binom{b+s}{s}}{\binom{a+b}{a}},
\end{equation}
we obtain, under the hypotheses of that lemma,
\begin{equation}\label{eq:subspaces-gap}
 \sum_{i=1}^{m}
 \frac{\binom{b_i+s}{s}}{\binom{a_i+b_i}{a_i}}\le1.
\end{equation}
For $s=0$, the construction leaves the original family unchanged and
\eqref{eq:subspaces-gap} is Proposition~\ref{thm:subspaces}.

\medskip
\noindent\textbf{Remark.}
For sets, the case $t=0$ and arbitrary $s$ also follows directly from the
classical Bollob\'as inequality. Indeed, for each
$S\in\binom{A_i}{s}$ replace $(A_i,B_i)$ by
$(A_i\setminus S,B_i\cup S)$. The enlarged family satisfies the ordinary
symmetric cross-intersection condition, and hence
\[
 \sum_{i=1}^m
 \frac{\binom{a_i}{s}}{\binom{a_i+b_i}{a_i-s}}\le1.
\]
Using~\eqref{eq:one-sided-identity} gives~\eqref{eq:subspaces-gap} in the
set case. For $t>0$, however, the common section need not leave all resulting subspaces
coordinate with respect to a single basis, so the general proof requires
Proposition~\ref{thm:subspaces}.

\subsection{Reduction to \texorpdfstring{$t=0$}{t=0} by a common section}
\label{subsec:sections}

To prove the subspace case of the Main Theorem, we shall use the following
generic-position lemma; see \cite[Lemma~26.14]{FT}.

\begin{lemma}\label{lem:common-section}
Let $V$ be a finite-dimensional vector space over an infinite field, and let
$\mathcal S$ be a finite family of proper subspaces of $V$. For every integer
$t$ with $0\le t\le\dim V$, there exists a subspace $W\subseteq V$ of
codimension $t$ such that
\[
 \dim(W\cap S)=\max\{\dim S-t,0\}
 \qquad(S\in\mathcal S).
\]
\end{lemma}

\begin{proof}[Proof of the Main Theorem]
We first prove the subspace case. Since $m\ge2$, the cross-intersection
conditions give $a_i,b_i\ge t+s+1$ for every $i$. For $t=0$, apply
\eqref{eq:subspaces-gap}. Assume $t>0$. By the scalar-extension argument at
the end of the proof of Theorem~\ref{thm:sandwich}, we may replace
$\mathbb F$ by $\overline{\mathbb F}$ if necessary and assume that it is
infinite. Let $V$ be the ambient space.

Consider the finite family
\[
 \mathcal S=\{A_i,B_i:1\le i\le m\}
              \cup\{A_i\cap B_j:1\le i,j\le m\}.
\]
Each $A_i$ and $B_i$ is a proper subspace: if $A_i=V$, then
$b_i\le t$, contrary to $b_i\ge t+s+1$, and similarly $B_i\ne V$.
Hence every member of $\mathcal S$ is proper.
By Lemma~\ref{lem:common-section}, there exists a subspace
$W\subseteq V$ of codimension $t$ such that
\[
 \dim(W\cap S)=\max\{\dim S-t,0\}
 \qquad(S\in\mathcal S).
\]
For $A'_i=A_i\cap W$ and $B'_i=B_i\cap W$, it follows that
$\dim A'_i=a_i-t\ge s+1$ and $\dim B'_i=b_i-t\ge s+1$.
Since $A'_i\cap B'_j=W\cap(A_i\cap B_j)$, we also have
\[
 A'_i\cap B'_i=\{0\},\qquad
 \dim(A'_i\cap B'_j)>s\quad(i\ne j).
\]
Applying~\eqref{eq:subspaces-gap} to these subspaces of $W$ gives
\[
 \sum_{i=1}^{m}
 \frac{\binom{(b_i-t)+s}{s}}
      {\binom{(a_i-t)+(b_i-t)}{a_i-t}}\le1,
\]
which is~\eqref{eq:main-bound}.

For the set case, let $X=\bigcup_{i=1}^m(A_i\cup B_i)$, and let
$\mathbb F^X$ have basis $\{e_x:x\in X\}$ over any field $\mathbb F$.
For $S\subseteq X$, set
$E_S=\operatorname{span}_{\mathbb F}\{e_x:x\in S\}$.
Then $\dim E_S=|S|$ and $E_A\cap E_B=E_{A\cap B}$. Hence the pairs
$(E_{A_i},E_{B_i})$ satisfy the subspace hypotheses of the theorem with
$a_i=|A_i|$ and $b_i=|B_i|$, and the set case follows from the subspace case.
\end{proof}

\section{Optimality and further inequalities}\label{sec:consequences}

Put
\[
 \alpha_i=a_i-t,\qquad \beta_i=b_i-t.
\]
Under the hypotheses of the Main Theorem, we have $\alpha_i,\beta_i\ge s+1$.

\subsection{The optimal bound for the F\"uredi weight sum}
\label{subsec:sharp-constant}

\begin{corollary}\label{cor:sharp-constant}
Under the hypotheses of the Main Theorem,
\begin{equation}\label{eq:sharp-constant}
 \sum_{i=1}^{m}\frac{1}{\binom{a_i+b_i-2t}{a_i-t}}
 \le\frac{2}{\binom{2s+2}{s+1}}.
\end{equation}
The constant is best possible for all nonnegative integers $t$ and $s$. For
$s\ge1$, equality holds if and only if $m=2$,
$a_1=b_1=a_2=b_2=t+s+1$, and $(A_2,B_2)=(B_1,A_1)$.
\end{corollary}

\begin{proof}
Put $c_s=\binom{2s+1}{s}$. Since $\beta_i\ge s+1$,
\[
 \binom{\beta_i+s}{s}\ge c_s.
\]
The Main Theorem therefore gives
\[
 c_s\sum_{i=1}^{m}\frac{1}{\binom{\alpha_i+\beta_i}{\alpha_i}}
 \le\sum_{i=1}^{m}
 \frac{\binom{\beta_i+s}{s}}
      {\binom{\alpha_i+\beta_i}{\alpha_i}}\le1.
\]
As $c_s=\frac12\binom{2s+2}{s+1}$, this proves~\eqref{eq:sharp-constant}.

To see that the constant is best possible, choose pairwise disjoint sets
$T,X,Y$ with $|T|=t$ and $|X|=|Y|=s+1$, and take
\[
 (A_1,B_1)=(T\cup X,T\cup Y),\qquad
 (A_2,B_2)=(T\cup Y,T\cup X).
\]
Their diagonal intersections have size $t$, their cross-intersections have
size $t+s+1$, and the left-hand side of~\eqref{eq:sharp-constant} equals
$2/\binom{2s+2}{s+1}$. The associated coordinate subspaces give equality over
any field.

Suppose now that $s\ge1$ and equality holds in~\eqref{eq:sharp-constant}.
The function $b\mapsto\binom{b+s}{s}$ is strictly increasing for
$b\ge s+1$. Equality in the preceding chain of inequalities thus gives
$\beta_i=s+1$ for every $i$. Interchanging $A_i$ and $B_i$ and applying the
same argument gives $\alpha_i=s+1$ for every $i$. Consequently,
\[
 \frac{m}{\binom{2s+2}{s+1}}=\frac{2}{\binom{2s+2}{s+1}},
\]
so $m=2$. The cross-intersection conditions, together with
$a_1=b_1=a_2=b_2=t+s+1$, imply $A_1=B_2$ and $A_2=B_1$, for sets as well
as for subspaces. Conversely, under the hypotheses of the corollary,
these dimensions and $m=2$ give equality directly.
\end{proof}

\subsection{Equality constructions for the strengthened inequality}
\label{subsec:equality-main}

The equality statement in Corollary~\ref{cor:sharp-constant} concerns
\eqref{eq:sharp-constant}, not the strengthened sum in
\eqref{eq:main-bound}. The latter admits equality for every prescribed
number of pairs.

\begin{proposition}\label{prop:arbitrary-m-equality}
For all integers $t,s\ge0$ and $m\ge2$, there is a strong
Bollob\'as $(t,s)$-system with exactly $m$ pairs attaining equality in
\eqref{eq:main-bound}. There is also a corresponding equality system of
coordinate subspaces over every field.
\end{proposition}

\begin{proof}
Put $r=s+1$. Choose pairwise disjoint sets $T,X_1,\ldots,X_m$ with
$|T|=t$ and $|X_i|=r$, and write
\[
 Y=X_1\sqcup\cdots\sqcup X_m.
\]
For $1\le i\le m$, take
\[
 A_i=T\cup X_i,\qquad B_i=T\cup(Y\setminus X_i).
\]
Then $A_i\cap B_i=T$, while
\[
 A_i\cap B_j=T\cup X_i\quad(i\ne j).
\]
Hence the diagonal intersections have size $t$ and the cross-intersections
have size $t+s+1$. Also $a_i-t=r$ and $b_i-t=(m-1)r$, so every summand in
\eqref{eq:main-bound} equals
\[
 \frac{\binom{(m-1)r+s}{s}}{\binom{mr}{r}}
 =\frac{\binom{mr-1}{r-1}}{\binom{mr}{r}}
 =\frac1m.
\]
Summing gives equality. Replacing all sets by their coordinate spans in
a vector space with basis indexed by $T\cup Y$ gives equality over every
field, since dimensions and intersections are preserved.
\end{proof}

Thus the constant $1$ in the main inequality cannot be decreased even
when $m$ is prescribed. For $s\ge1$ and $m>2$, these examples do not attain
\eqref{eq:sharp-constant}: the latter has the more restrictive equality
conditions in Corollary~\ref{cor:sharp-constant}.

\medskip
\noindent\textbf{Remark.}
The equality problem has the following natural design-theoretic interpretation.
Let $T$ and $Y$ be disjoint with $|T|=t$ and $|Y|=v$, and let
$\mathcal C\subseteq\binom{Y}{w}$. For
\[
 (A_C,B_C)=\bigl(T\cup C,\,T\cup(Y\setminus C)\bigr),
 \qquad C\in\mathcal C,
\]
the strong Bollob\'as $(t,s)$ condition is equivalent to
$|C\cap D|\le w-s-1$ for distinct $C,D\in\mathcal C$. Counting
$(w-s)$-subsets contained in the blocks gives
\[
 |\mathcal C|\binom{w}{s}\le\binom{v}{w-s},
\]
which is precisely the set case of the Main Theorem for this construction.
Equality holds exactly when every $(w-s)$-subset of $Y$ lies in one block,
i.e., when $\mathcal C$ is a Steiner system $S(w-s,w,v)$. In particular,
the affine hyperplanes of $\mathbb F_2^3$ form an $S(3,4,8)$ and give an
equality example with $s=1$ and $m=14$.

\subsection{Two-sided exchange inequalities}
\label{subsec:two-sided}

Lemma~\ref{lem:exchange} gives the following two-parameter family of
inequalities.

\begin{corollary}\label{cor:two-sided}
Under the hypotheses of the Main Theorem, for all nonnegative integers $p,q$ with
$p+q\le s$,
\begin{equation}\label{eq:two-sided}
 \sum_{i=1}^{m}
 \frac{H_{p,q}(\alpha_i,\beta_i)}
      {\binom{\alpha_i+\beta_i}{\alpha_i-p+q}}\le1,
\end{equation}
where $H_{p,q}$ is defined in~\eqref{eq:exchange-weight}.
In particular, for every nonnegative integer $r$ with $2r\le s$,
\begin{equation}\label{eq:symmetric-exchange}
 \sum_{i=1}^{m}
 \frac{\displaystyle\sum_{j=0}^{r}
       \binom{\alpha_i}{j}\binom{\beta_i}{j}}
      {\binom{\alpha_i+\beta_i}{\alpha_i}}\le1.
\end{equation}
\end{corollary}

\begin{proof}
For subspaces, apply the scalar extension and common section argument in
the proof of the subspace case of the Main Theorem. It gives subspace pairs
$(A'_i,B'_i)$ of dimensions $\alpha_i,\beta_i$ such that
\[
 A'_i\cap B'_i=\{0\},\qquad
 \dim(A'_i\cap B'_j)>s\quad(i\ne j).
\]
For $t=0$, simply take $A'_i=A_i$ and $B'_i=B_i$.
Lemma~\ref{lem:exchange} gives~\eqref{eq:two-sided}. For sets, use the
coordinate subspace construction in the proof of the set case of the Main Theorem
and apply the subspace result. Finally, take $p=q=r$ and reindex the sum
in~\eqref{eq:exchange-weight} to obtain~\eqref{eq:symmetric-exchange}.
\end{proof}

The choices $(p,q)=(s,0)$ and $(0,s)$ recover the main inequalities
and their versions with $A_i,B_i$ interchanged.

\section*{Statement on the use of AI}

During the preparation of this manuscript, the authors developed the
exterior-algebra approach and the overall strategy of the paper. OpenAI Codex
was used as an auxiliary tool for exploratory discussion and English-language
editing. The authors checked every proof in detail, revised the arguments,
notation, and organization, and finalized the presentation of the results. The
authors take full responsibility for the content of the final manuscript.

\enlargethispage{2\baselineskip}


\begin{thebibliography}{99}
\bibitem{Alon}
Noga Alon,
\emph{An extremal problem for sets with applications to graph theory},
J. Combin. Theory Ser. A \textbf{40} (1985), no.~1, 82--89.
\href{https://doi.org/10.1016/0097-3165(85)90048-2}{doi:\nolinkurl{10.1016/0097-3165(85)90048-2}}.

\bibitem{BF}
L\'aszl\'o Babai and P\'eter Frankl,
\emph{Linear Algebra Methods in Combinatorics, with Applications
to Geometry and Computer Science},
Preliminary Version~2, Department of Computer Science,
University of Chicago, 1992.

\bibitem{Bollobas}
B\'ela Bollob\'as,
\emph{On generalized graphs},
Acta Math. Acad. Sci. Hungar. \textbf{16} (1965), 447--452.
\href{https://doi.org/10.1007/BF01904851}{doi:\nolinkurl{10.1007/BF01904851}}.

\bibitem{Frankl}
Peter Frankl,
\emph{An extremal problem for two families of sets},
European J. Combin. \textbf{3} (1982), no.~2, 125--127.
\href{https://doi.org/10.1016/S0195-6698(82)80025-5}{doi:\nolinkurl{10.1016/S0195-6698(82)80025-5}}.

\bibitem{FT}
Peter Frankl and Norihide Tokushige,
\emph{Extremal Problems for Finite Sets},
Student Mathematical Library, vol.~86,
American Mathematical Society, Providence, RI, 2018.
\href{https://doi.org/10.1090/stml/086}{doi:\nolinkurl{10.1090/stml/086}}.

\bibitem{Furedi}
Zolt\'an F\"uredi,
\emph{Geometrical solution of an intersection problem for two hypergraphs},
European J. Combin. \textbf{5} (1984), no.~2, 133--136.
\href{https://doi.org/10.1016/S0195-6698(84)80026-8}{doi:\nolinkurl{10.1016/S0195-6698(84)80026-8}}.

\bibitem{Heg}
G\'abor Heged\"us,
\emph{On F\"uredi's conjecture},
Acta Math. Hungar. \textbf{174} (2024), no.~1, 244--246.
\href{https://doi.org/10.1007/s10474-024-01461-8}{doi:\nolinkurl{10.1007/s10474-024-01461-8}}.

\bibitem{HegPreprint}
G\'abor Heged\"us,
\emph{About F\"uredi's conjecture},
arXiv:2406.05841v1, 2024.
\url{https://arxiv.org/abs/2406.05841v1}.

\bibitem{HF}
G\'abor Heged\"us and P\'eter Frankl,
\emph{Variations on the Bollob\'as set-pair theorem},
European J. Combin. \textbf{120} (2024), article~103983.
\href{https://doi.org/10.1016/j.ejc.2024.103983}{doi:\nolinkurl{10.1016/j.ejc.2024.103983}}.

\bibitem{KKK}
Dong Yeap Kang, Jaehoon Kim, and Younjin Kim,
\emph{On the Erd\H{o}s--Ko--Rado theorem and the Bollob\'as theorem
for $t$-intersecting families},
European J. Combin. \textbf{47} (2015), 68--74.
\href{https://doi.org/10.1016/j.ejc.2015.01.009}{doi:\nolinkurl{10.1016/j.ejc.2015.01.009}}.

\bibitem{LFW}
Zhiyi Liu, Lihua Feng, and Tingzeng Wu,
\emph{Bollob\'as-type inequalities for subspaces via weight invariance},
arXiv:2603.25007v1, 2026.
\url{https://arxiv.org/abs/2603.25007v1}.

\bibitem{Lovasz}
L\'aszl\'o Lov\'asz,
\emph{Flats in matroids and geometric graphs},
in \emph{Combinatorial Surveys: Proceedings of the Sixth British
Combinatorial Conference}, Academic Press, London, 1977, 45--86.

\bibitem{Lubell}
David Lubell,
\emph{A short proof of Sperner's lemma},
J. Combin. Theory \textbf{1} (1966), no.~2, 299.
\href{https://doi.org/10.1016/S0021-9800(66)80035-2}{doi:\nolinkurl{10.1016/S0021-9800(66)80035-2}}.

\bibitem{Meshalkin}
L.~D. Meshalkin,
\emph{Generalization of Sperner's theorem on the number of subsets
of a finite set},
Theory Probab. Appl. \textbf{8} (1963), no.~2, 203--204.
\href{https://doi.org/10.1137/1108023}{doi:\nolinkurl{10.1137/1108023}}.

\bibitem{SW}
Alex Scott and Elizabeth Wilmer,
\emph{Combinatorics in the exterior algebra and the Bollob\'as Two Families Theorem},
J. London Math. Soc. \textbf{104} (2021), no.~4, 1812--1839.
\href{https://doi.org/10.1112/jlms.12484}{doi:\nolinkurl{10.1112/jlms.12484}}.



\bibitem{Talbot}
John Talbot,
\emph{A new Bollob\'as-type inequality and applications to $t$-intersecting families of sets},
Discrete Math. \textbf{285} (2004), nos.~1--3, 349--353.
\href{https://doi.org/10.1016/j.disc.2004.04.002}{doi:\nolinkurl{10.1016/j.disc.2004.04.002}}.

\bibitem{Tuza87}
Zsolt Tuza,
\emph{Inequalities for two set systems with prescribed intersections},
Graphs Combin. \textbf{3} (1987), no.~1, 75--80.
\href{https://doi.org/10.1007/BF01788531}{doi:\nolinkurl{10.1007/BF01788531}}.

\bibitem{TuzaSurvey}
Zsolt Tuza,
\emph{Applications of the set-pair method in extremal hypergraph theory},
in \emph{Extremal Problems for Finite Sets} (Visegr\'ad, 1991),
Bolyai Society Mathematical Studies, vol.~3,
J\'anos Bolyai Mathematical Society, Budapest, 1994, 479--514.

\bibitem{WLLF}
Yongjiang Wu, Yongtao Li, Lu Lu, and Lihua Feng,
\emph{Subspace variations of the weighted skew Bollob\'as theorem},
arXiv:2603.02698v1, 2026.
\url{https://arxiv.org/abs/2603.02698v1}.

\bibitem{Yamamoto}
Koichi Yamamoto,
\emph{Logarithmic order of free distributive lattice},
J. Math. Soc. Japan \textbf{6} (1954), nos.~3--4, 343--353.
\href{https://doi.org/10.2969/jmsj/00630343}{doi:\nolinkurl{10.2969/jmsj/00630343}}.

\bibitem{YuEtAl}
Wenjun Yu, Xiangliang Kong, Yuanxiao Xi, Xiande Zhang, and Gennian Ge,
\emph{Bollob\'as-type theorems for hemi-bundled two families},
European J. Combin. \textbf{100} (2022), article~103438.
\href{https://doi.org/10.1016/j.ejc.2021.103438}{doi:\nolinkurl{10.1016/j.ejc.2021.103438}}.

\bibitem{Yue}
Erfei Yue,
\emph{Some new Bollob\'as-type inequalities},
Discrete Math. \textbf{349} (2026), no.~5, article~114948.
\href{https://doi.org/10.1016/j.disc.2025.114948}{doi:\nolinkurl{10.1016/j.disc.2025.114948}}.

\bibitem{Zhu}
C.~Z. Zhu,
\emph{On two set-systems with restricted cross-intersections},
European J. Combin. \textbf{16} (1995), no.~6, 655--658.
\href{https://doi.org/10.1016/0195-6698(95)90047-0}{doi:\nolinkurl{10.1016/0195-6698(95)90047-0}}.

\end{thebibliography}
\end{document}